\documentclass[12pt]{amsart} %

\usepackage[english]{babel}
\usepackage[utf8]{inputenc}
\usepackage{amsmath, amsfonts, amssymb, amsthm}
\usepackage{mathtools}
\usepackage{listings}
\usepackage[left=2.5cm,right=2.5cm]{geometry}
\usepackage{array}
\usepackage{verbatim}
\usepackage{hyperref}
\usepackage{xcolor}

\newtheorem{theorem}{Theorem}
\newtheorem*{theorem*}{Theorem}
\newtheorem{lemma}{Lemma}
\newtheorem{proposition}{Proposition}
\newtheorem{corollary}{Corollary}

\theoremstyle{definition}

\newtheorem{remark}{Remark}

\begin{document}

\title[Existence of non-cyclic abelian varieties]{Existence of simple non-cyclic abelian varieties over arbitrary finite fields and of a given dimension $g>1$}
\author{Alejandro J. Giangreco Maidana}
\date{July 9, 2025}



\begin{abstract}
Vl{\u a}du{\c t} characterized in 1999 the set of finite fields $k$ such that all elliptic curves defined over $k$ have a cyclic group of rational points. Under the conjecture of infinitely many Mersenne primes, this set is infinite. In these notes we prove that there is no a finite field $k$ such that all the simple abelian varieties defined over $k$ of dimension $g>1$ have a cyclic group of rational points.
\end{abstract}

\maketitle

\section{Introduction}
Abelian varieties over finite fields whose group of rational points is cyclic (called \emph{cyclic} varieties for short) play an important role in both theory and applications. For example, they are used in cryptography where the discrete logarithm problem is exploited. The statistics on cyclic varieties are related to a positive characteristic analogue of the Cohen-Lenstra heuristics (\cite{CohenLenstra1984}), which, roughly speaking, states that the odd part of random abelian groups tend to be cyclic. The concept of cyclicity first emerged in discussions about the conjectures of Lang and Trotter (\cite{lang1977}) and in Serre's ``Résumé des cours 1977-78'' (\cite{Serre1986Resume}):  given an elliptic curve defined over the rational numbers, we are interested in the set of primes such that the reduction is a cyclic elliptic curve. This question was also studied by Gupta, Murty and others. Generalizations to higher dimensions were also done.

In \cite[Theorem 4.1]{VLADUT199913}, Vl{\u a}du{\c t} characterized the set $\mathcal{C}_1$ of prime powers $q$ for which the elliptic curves defined over $\mathbb{F}_q$ all have a cyclic group of rational points. It is not known whether the set $\mathcal{C}_1$ is infinite or not. However, we have
\[
\{q : q-1 \text{ is a Mersenne prime}\} \subset \mathcal{C}_1,
\]
so $\mathcal{C}_1$ is conjecturally infinite.

In higher dimension there is no such field. Indeed, it is enough to consider the product of an elliptic curve by itself. It is natural though to pose the same question if we restrict ourselves to simple abelian varieties.
In these short notes, we prove that in dimension $g > 1$, there is no such a field. 

\begin{theorem}\label{thm:nonexistence}
Let $g>1$. For every finite field $\mathbb{F}_q$ there is a simple and ordinary $g$-dimensional abelian variety defined over $\mathbb{F}_q$ with a non-cyclic $3$-primary component of its group of rational points, except in cases $(g,\mathbb{F}_q) \in \{(2,\mathbb{F}_2), (2,\mathbb{F}_3), (3,\mathbb{F}_2)\}$. 
In the exceptional cases, there is a $g$-dimensional simple abelian variety defined over $\mathbb{F}_q$ with a non-cyclic $2$-primary component of its group of rational points but all the $\ell$-primary components are cyclic for every prime $\ell>2$.
\end{theorem}
The following is a direct consequence.
\begin{corollary}
Let $g>1$. There is no finite field $\mathbb{F}_q$ such that all the simple $g$-dimensional abelian varieties defined over $\mathbb{F}_q$ have a cyclic group of rational points.    
\end{corollary}

\begin{remark}
For $q$ big enough (asymptotically and depending on $g$), Theorem \ref{thm:nonexistence} follows from \cite[Theorem 2.1]{GIANGRECOMAIDANA2020101628}, \cite{GIANGRECOMAIDANA2021101703} and \cite[Theorem 2]{HOWE2002139}.
\end{remark}

Thus, it would be natural to study the problem locally, in the following sense.
One could think that for a set of primes $S$, we can find a set $\mathcal{C}_{g,S}$ of prime powers $q$ such that all the simple $g$-dimensional abelian varieties defined over $\mathbb{F}_q$ are cyclic except possibly at the primes in $S$. Again, it follows from  \cite[Theorem 2.1]{GIANGRECOMAIDANA2020101628}, \cite{GIANGRECOMAIDANA2021101703} and \cite[Theorem 2]{HOWE2002139} that such $\mathcal{C}_{g,S}$ cannot be infinite. 
It would be interesting to explore the maximal element of, for instance,  $\mathcal{C}_{g,\{2,3\}}$, but we let that for an upcoming research.

\section{Proof}
\subsection*{Strategy}
For any $g>1$ and any $q$, we will search for an abelian variety defined over $\mathbb{F}_q$ of dimension $g$ that is not cyclic. We are going to find it by finding an isogeny class that is not cyclic. We have a criterion for this (see Section \ref{sec:cylicity_criterion} below). In order to do that, every $q$-Weil polynomial $f$ can be converted to another simpler polynomial $h$. The converse holds; every such $h$ can produce an $f$ for at least a family of values of $q$, provided that some conditions are verified. The cyclicity criterion can also be converted in terms of $h$. 
The most technical part is to find the polynomial $h$ with the required arithmetic conditions for every $g$ and every family of $q$'s.
We are going to give the polynomials $h$ explicitly for dimension $\leq 13$.
There will be at least three polynomials $h$ for every $g$, since they depend on the value of $q$ modulo $3$.
Weil polynomials are given directly in the exceptional cases.
A similar strategy used in \cite{HOWE2002139} (that uses a result of \cite{HSU199685}) allows us to prove the existence of non-cyclic isogeny classes (equivalently, the existence of polynomials $h$ with the required arithmetic properties) for $g\geq 14$. 

We start by giving some basic information about abelian varieties over finite fields. 
After that we explore the simpler polynomials $h$. 
Then we study the cyclicity of the varieties depending on the Weil polynomial and the polynomials $h$. 
Finally we prove the existence of the required polynomials $h$\footnote{At the end of the .tex file, the list of polynomials can be found commented, as well as some Magma (\cite{MR1484478}) codes that can be use to verify results.}.

\subsection*{Abelian varieties over finite fields}
We refer the reader to \cite{mumford1970abelian} for the general theory of abelian varieties, and to \cite{Waterhouse1969} for abelian varieties over finite fields.
Let $q=p^r$ be a power of a prime, and let $\mathbb{F}_q$ be a finite field with $q$ elements. To every abelian variety $A$ defined over $\mathbb{F}_q$ we can associate its \emph{Weil polynomial}: the characteristic polynomial of its Frobenius endomorphism acting on its Tate module. This polynomial is an invariant of the isogeny class. Moreover, from the Honda-Tate theory, it defines completely the isogeny class. Thus, we can consider the Weil polynomial $f_\mathcal{A}$ of an isogeny class $\mathcal{A}$. 
The Weil polynomial $f_\mathcal{A}$ is a $q$\emph{-Weil polynomial}, i.e. it is a monic polynomial with integer coefficients whose roots have all absolute value $\sqrt{q}$. A $q$-Weil polynomial is said to be \emph{ordinary} if its central term is coprime with $q$.
Not every irreducible $q$-Weil polynomial is the Weil polynomial of a simple isogeny class of abelian varieties. However, when restricted to ordinary isogeny classes, there is a bijection between isogeny classes of simple ordinary abelian varieties defined over $\mathbb{F}_q$ and irreducible ordinary $q$-Weil polynomials (this is the ordinary Honda--Tate theory, see \cite[Theorem 3.3]{Howe1995}).

Over real numbers, $q$-Weil polynomials have the form
\begin{align}\label{eq:Weil_p_real}
\prod_{i=1}^g (x^2 - \alpha_i x + q), \quad \alpha_i\in \mathbb{R} \text{ and } |\alpha_i|\leq 2\sqrt{q}, 
\end{align}
provided that the roots are complex numbers. If a $q$-Weil polynomial has real roots, then it has the form $(x \pm\sqrt{q})^2$ or $(x^2-q)^2$ (see \cite[p.~528]{Waterhouse1969}). 

\subsection*{A ``reduction'' of Weil polynomials}
To every Weil polynomial of the form of equation (\ref{eq:Weil_p_real}), we attach to it the polynomial 
\begin{align}\label{eq:h_form}
\prod_{i=1}^{g} (x-\alpha_i),
\end{align}
that in this paper we usually call it $h$. On the other side, every polynomial of the form of equation (\ref{eq:h_form}) produces a $q$-Weil polynomial $\prod_{i=1}^{g} (x^2- \alpha_i x+q) \in \mathbb{Z}[x]$, provided that $2\sqrt{q} \geq \max \{\mid \alpha_i \mid\}$. 
Moreover
\begin{lemma}
The polynomial $f(x)=\prod_{i=1}^{g} (x^2- \alpha_i x+q)$ is irreducible over $\mathbb{Q}$ if and only if the polynomial $h(x)=\prod_{i=1}^{g} (x-\alpha_i)$ is irreducible over $\mathbb{Q}$. The Weil polynomial $f$ is ordinary if and only if $h(0)$ is coprime with $q$.
\end{lemma}
\begin{proof}
The proof of the first part is very similar as \cite[Lemma 2.1]{berardini2022weil}. For the last part, note that middle term of $f$ equals the constant term of $h$ plus a multiple of $q$. Thus, to check the ordinariness of the Weil polynomial $f$, we need to check that the constant term of $h$ is coprime with $q$.    
\end{proof}
We recall that if the Weil polynomial is ordinary, then it is the Weil polynomial of an abelian variety.

\subsection*{(Non-)Cyclic isogeny classes}\label{sec:cylicity_criterion}
We say that an isogeny class of abelian varieties is \emph{cyclic} if all its varieties are cyclic. We denote by $\widehat{z}$ the quotient of an integer $z$ by its radical. As result of the author's work, the following is the key of the proof:
\begin{theorem*}[2019, \cite{GIANGRECOMAIDANA2019139}]\label{th:weil_polynomial_criterion}
Let $\mathcal{A}$ be a $g$-dimensional $\mathbb{F}_q$-isogeny class of abelian varieties corresponding to the Weil polynomial $f_\mathcal{A}(x)$. Then $\mathcal{A}$ is cyclic if and only if  $f'_\mathcal{A}(1)$ is coprime with $\widehat{f_\mathcal{A}(1)}$.
\end{theorem*}
In other words, an isogeny class has a non cyclic variety if there is a prime $\ell$ such that $\ell\vert f'(0)$ and $\ell^2 \vert f(0)$.
In this case, we say that the isogeny class is non-$\ell$-cyclic.
We will prove Theorem \ref{thm:nonexistence} by proving that for each $\mathbb{F}_q$ there is at least one non-cyclic isogeny class. 
In order to apply the cyclicity criterion, we express $f(1)$ and $f'(1)$ in function of $h$ and $q$:

\begin{align*}
f(1)&=\prod_{i=1}^{g} (1+q-\alpha_i) = h(1+q), \quad \text{and}\\
f'(1)&=\sum_{i=1}^g (2-\alpha_i) \prod_{j\neq i} (1+q-\alpha_j)=\sum_{i=1}^g (1-q+1+q-\alpha_i) \prod_{j\neq i} (1+q-\alpha_j) \\
&=(1-q)h'(1+q)+gh(1+q).    
\end{align*}

Let $\ell$ be an integer prime such that
\begin{itemize}
\item $\ell \vert q+1$, then
\begin{align*}
f(1) &\equiv (1+q)h'(0) + h(0) \pmod{\ell^2} \\
f'(1) &\equiv (1-q)h'(0)+gh(0) \pmod{\ell}
\end{align*}
In order to have both equivalences equal to $0$, it is enough to have $\ell\vert h'(0)$ and $\ell^2 \vert h(0)$.
\item $\ell \vert q-1$, then
\begin{align*}
f(1) &= h(q+1) = \prod_{i=1}^{g} (q-1+2-\alpha_i) \equiv  (q-1)h'(2)+h(2) \pmod{\ell^2} \\
f'(1) &\equiv gh(q+1) \equiv gh(2) \pmod{\ell}  
\end{align*}
In order to have both equivalences equal to $0$, it is enough to have $\ell\vert h'(2)$ and $\ell^2 \vert h(2)$.
\end{itemize}
As in general we want to prove the non-$3$-cyclicity, we are going to use always $\ell=3$, so it remains only one case
\begin{itemize}
\item $q=3^r$, then
\begin{align*}
f(1) &= h(q+1) \equiv  3h'(1)+h(1) \pmod{9} \\
f'(1) &\equiv (1+q)h'(1)+gh(1) \pmod{3}    
\end{align*}
In order to have both equivalences equal to $0$, it is enough to have $3\vert h'(1)$ and $9 \vert h(1)$. 
\end{itemize}

Thus, in all the cases, we need $3\vert h'(n)$ and $9 \vert h(n)$ for $n\in\{0,1,2\}$ as necessary and sufficient conditions for the non-$3$-cyclicity.
Table \ref{tab:cases_q_summary} summarizes the conditions needed for $h$. In all cases we need $h$ to be an irreducible polynomial $\in \mathbb{Z}[x]$ with all real roots and $h(0)$ coprime with $q$.
Note that if the maximal root of $h$ in absolute value is bigger than the given in the last column of the table, this $h$ is still valid for bigger values of $q$.

\begin{table}[ht]
\renewcommand\arraystretch{1.25}
    \centering
    \begin{tabular}{c|c|c|c}
      Case & Conditions for non-$3$-cyclicity & Min $q$ & Max root for $h$ (in absolute value)   \\ \hline \hline
      $3 \vert q+1$   &  $3\vert h'(0)$ and $9 \vert h(0)$ & $2$ & $2\sqrt{2} \approx 2.828$ \\ \hline
      $3 \vert q-1$   &  $3\vert h'(2)$ and $9 \vert h(2)$ & $4$ & $2\sqrt{4} = 4$ \\ \hline
      $3^r = q$   &  $3\vert h'(1)$ and $9 \vert h(1)$ & $3$ & $2\sqrt{3} \approx 3.464$ \\
    \end{tabular}
    \caption{Conditions on $h$ according to each case on $q$.}
    \label{tab:cases_q_summary}
\end{table}

\subsection*{Existence of required polynomials}
\subsubsection*{Dimensions smaller than 14}
In this subsection, we prove the following result.
\begin{proposition}
There exist a non-$(3)$-cyclic isogeny class of $g$-dimensional abelian varieties defined over $\mathbb{F}_q$ for any $g$ such that $2\leq g\leq 13$, except in the cases $(g,\mathbb{F}_q) \in \{(2,\mathbb{F}_2), (2,\mathbb{F}_3), (3,\mathbb{F}_2)\}$. In the exceptional cases, the non-$(\ell)$-cyclic classes are only for $\ell=2$, and the Weil polynomials of the complete list of non-cyclic isogeny classes are given in Table \ref{tab:exceptions}.
\end{proposition}

\begin{table}[ht]
\renewcommand\arraystretch{1.25}
    \centering
    \begin{tabular}{c|c|c}
Dimension & Base field & Weil polynomial  \\ \hline
2 & $\mathbb{F}_2$ & $x^4-x^2+4$  \\ \hline
2 & $\mathbb{F}_3$ & $x^4-2x^3+2x^2-6x+9$ \\ \hline
2 & $\mathbb{F}_3$ & $x^4-x^3-2x^2-3x+9$ \\ \hline
2 & $\mathbb{F}_3$ & $x^4-x^3+2x^2-3x+9$ \\ \hline
2 & $\mathbb{F}_3$ & $x^4-6x^2+9$ \\ \hline
2 & $\mathbb{F}_3$ & $x^4-6x^2+9$ \\ \hline
2 & $\mathbb{F}_3$ & $x^4-2x^2+9$ \\ \hline
2 & $\mathbb{F}_3$ & $x^4+x^3-2x^2+3x+9$ \\ \hline
2 & $\mathbb{F}_3$ & $x^4+x^3+2x^2+3x+9$ \\ \hline
2 & $\mathbb{F}_3$ & $x^4+2x^3+2x^2+6x+9$ \\ \hline
3 & $\mathbb{F}_2$ & $x^6 - x^5 + x^4 -3x^3 + 2x^2 -4x + 8$ \\ \hline
3 & $\mathbb{F}_2$ & $x^6  - x^4 -2x^3 - 2x^2 + 8$ \\ \hline
3 & $\mathbb{F}_2$ & $x^6  - x^4 +2x^3 - 2x^2 + 8$ \\ \hline
3 & $\mathbb{F}_2$ & $x^6 + x^5 + x^4 +3x^3 + 2x^2 +4x + 8$ \\ \hline
    \end{tabular}
    \caption{These isogeny classes are non-$2$-cyclic. The rest of isogeny classes of the corresponding dimensions and base field, are all cyclic. Note that there are ordinary and non-ordinary isogeny classes in the list.}
    \label{tab:exceptions}
\end{table}

Tables \ref{tab:poly_h_1} and \ref{tab:poly_h_2} give the list of polynomials $h$. In general, they can be applied for any $q$ (according to the cases of $q \pmod{3}$), to obtain the respective Weil polynomial.
Note that for $g=2$, the polynomial $h=x^{2} -x -11$ have a good bound, but its constant term is not always coprime with the $q$ for which $3|q-1$ (e.g. $q=121$), so we have another $h$ for that case.
The exceptions are precisely the exceptional cases of the proposition, where the bounds on the absolute values of the roots of some of the polynomials $h$ do not allow to use them (they can be seen at the beginning of Table \ref{tab:poly_h_1}). In these cases, an exhaustive search on the LMFDB database (\cite{lmfdb}) shows that the only non cyclic varieties are the given in Table \ref{tab:exceptions}.

\begin{table}[ht]
\renewcommand\arraystretch{1.25}
{\tiny
    \centering
    \begin{tabular}{c|>{\centering\arraybackslash}m{6cm}|c|c|c|c}
 $g$ & $h$ & $\ell|h(0)$ & Case $q$ & Max real root & Valid for $q\geq$ \\ \hline\hline
 $2$ & $x^{2} -x -11 $ & $[ 11 ]$ & $3|q-1$ & $3.854$ & $4$ \\ \hline
 $2$ & $x^{2} +2x -17 $ & $[ 17 ]$ & $3|q-1$ & $5.243$ & $7$ \\ \hline
 $2$ & $x^{2} +x -11 $ & $[ 11 ]$ & $3^r=q$ & $3.854$ & $9$ \\ \hline
 $2$ & $x^{2} -18 $ & $[ 2, 3 ]$ & $3|q+1$ & $4.243$ & $5$ \\ \hline\hline
 $3$ & $x^{3} -9x + $ & $[]$ & $3|q-1$ & $3.054$ & $4$ \\ \hline
 $3$ & $x^{3} -3x^{2} -6x +17 $ & $[ 17 ]$ & $3^r=q$ & $3.227$ & $3$ \\ \hline
 $3$ & $x^{3} -4x^{2} -3x +9 $ & $[ 3 ]$ & $3|q+1$ & $4.204$ & $5$ \\
\hline\hline
 $4$ & $x^{4} -2x^{3} -6x^{2} +7x + $ & $[]$ & $3|q-1$ & $3.165$ & $4$ \\ \hline
 $4$ & $x^{4} -4x^{3} -8x^{2} +48x -46 $ & $[ 2, 23 ]$ & $3^r=q$ & $3.454$ & $3$
\\ \hline
 $4$ & $x^{4} -x^{3} -8x^{2} +6x +9 $ & $[ 3 ]$ & $3|q+1$ & $2.682$ & $2$ \\
\hline\hline
 $5$ & $x^{5} -10x^{3} +x^{2} +18x - $ & $[]$ & $3|q-1$ & $2.843$ & $4$ \\
\hline
 $5$ & $x^{5} +4x^{4} -4x^{3} -22x^{2} -x +13 $ & $[ 13 ]$ & $3^r=q$ & $3.389$ &
$3$ \\ \hline
 $5$ & $x^{5} -10x^{3} +x^{2} +21x -9 $ & $[ 3 ]$ & $3|q+1$ & $2.677$ & $2$ \\
\hline\hline
 $6$ & $x^{6} -12x^{4} +34x^{2} +2x -3 $ & $[ 3 ]$ & $3|q-1$ & $2.786$ & $4$ \\
\hline
 $6$ & $x^{6} -12x^{4} +32x^{2} -x -11 $ & $[ 11 ]$ & $3^r=q$ & $2.895$ & $3$ \\
\hline
 $6$ & $x^{6} -12x^{4} +34x^{2} -9 $ & $[ 3 ]$ & $3|q+1$ & $2.789$ & $2$ \\
\hline\hline
 $7$ & $x^{7} -14x^{5} +56x^{3} -2x^{2} -55x - $ & $[]$ & $3|q-1$ & $2.805$ &
$4$ \\ \hline
 $7$ & $x^{7} -14x^{5} +56x^{3} -4x^{2} -58x + $ & $[]$ & $3^r=q$ & $2.863$ &
$3$ \\ \hline
 $7$ & $x^{7} -14x^{5} +56x^{3} -2x^{2} -57x +9 $ & $[ 3 ]$ & $3|q+1$ & $2.791$
& $2$ \\ \hline\hline
 $8$ & $x^{8} +8x^{7} +12x^{6} -39x^{5} -84x^{4} +61x^{3} +122x^{2} -41x + $ &
$[]$ & $3|q-1$ & $3.683$ & $4$ \\ \hline
 $8$ & $x^{8} -16x^{6} +81x^{4} +x^{3} -129x^{2} -2x + $ & $[]$ & $3^r=q$ &
$2.765$ & $3$ \\ \hline
 $8$ & $x^{8} -16x^{6} +81x^{4} +x^{3} -130x^{2} -3x +9 $ & $[ 3 ]$ & $3|q+1$ &
$2.758$ & $2$ \\ \hline\hline
 $9$ & $x^{9} +9x^{8} +18x^{7} -42x^{6} -144x^{5} +37x^{4} +298x^{3} +15x^{2}
-164x - $ & $[]$ & $3|q-1$ & $3.773$ & $4$ \\ \hline
 $9$ & $x^{9} -18x^{7} +108x^{5} +x^{4} -240x^{3} -9x^{2} +146x +2 $ & $[ 2 ]$ &
$3^r=q$ & $2.788$ & $3$ \\ \hline
 $9$ & $x^{9} -18x^{7} +108x^{5} +x^{4} -240x^{3} -9x^{2} +147x +9 $ & $[ 3 ]$ &
$3|q+1$ & $2.776$ & $2$ \\ \hline
    \end{tabular}
    \caption{List of polynomials $h$ (first part).}
    \label{tab:poly_h_1}
}
\end{table}

\begin{table}[ht]
\renewcommand\arraystretch{1.25}
{\tiny
    \centering
    \begin{tabular}{c| >{\centering\arraybackslash}m{6cm}|c|c|c|c}
 $g$ & $h$ & $\ell|h(0)$ & Case $q$ & Max real root & Valid for $q\geq$ \\ \hline\hline
 $10$ & $x^{10} +10x^{9} +25x^{8} -40x^{7} -210x^{6} -28x^{5} +510x^{4}
+201x^{3} -412x^{2} -113x - $ & $[]$ & $3|q-1$ & $3.826$ & $4$ \\ \hline
 $10$ & $x^{10} -20x^{8} +140x^{6} -400x^{4} +x^{3} +400x^{2} -7x -88 $ & $[ 2,
11 ]$ & $3^r=q$ & $2.810$ & $3$ \\ \hline
 $10$ & $x^{10} -20x^{8} +140x^{6} -400x^{4} +x^{3} +400x^{2} -6x -63 $ & $[ 3,
7 ]$ & $3|q+1$ & $2.796$ & $2$ \\ \hline\hline
 $11$ & $x^{11} -22x^{9} +176x^{7} -615x^{5} +871x^{3} +x^{2} -338x -9 $ & $[ 3
]$ & $3|q-1$ & $2.795$ & $4$ \\ \hline
 $11$ & $x^{11} -22x^{9} +176x^{7} -616x^{5} +880x^{3} +x^{2} -352x -5 $ & $[ 5
]$ & $3^r=q$ & $2.801$ & $3$ \\ \hline
 $11$ & $x^{11} -22x^{9} +176x^{7} -616x^{5} +x^{4} +880x^{3} -7x^{2} -351x +9 $
& $[ 3 ]$ & $3|q+1$ & $2.805$ & $2$ \\ \hline\hline
 $12$ & $x^{12} -12x^{11} +42x^{10} +20x^{9} -369x^{8} +360x^{7} +1036x^{6}
-1464x^{5} -1185x^{4} +1765x^{3} +519x^{2} -543x + $ & $[]$ & $3|q-1$ & $3.791$
& $4$ \\ \hline
 $12$ & $x^{12} -24x^{10} +216x^{8} -896x^{6} +x^{5} +1681x^{4} -10x^{3}
-1159x^{2} +20x +125 $ & $[ 5 ]$ & $3^r=q$ & $2.805$ & $3$ \\ \hline
 $12$ & $x^{12} -24x^{10} +216x^{8} -896x^{6} +x^{5} +1681x^{4} -10x^{3}
-1160x^{2} +21x +117 $ & $[ 3, 13 ]$ & $3|q+1$ & $2.810$ & $2$ \\ \hline\hline
 $13$ & $x^{13} -26x^{11} +260x^{9} -1248x^{7} +x^{6} +2912x^{5} -11x^{4}
-2912x^{3} +28x^{2} +833x -9 $ & $[ 3 ]$ & $3|q-1$ & $2.810$ & $4$ \\ \hline
 $13$ & $x^{13} -26x^{11} +260x^{9} -1248x^{7} +x^{6} +2913x^{5} -12x^{4}
-2921x^{3} +37x^{2} +847x -23 $ & $[ 23 ]$ & $3^r=q$ & $2.807$ & $3$ \\ \hline
 $13$ & $x^{13} -26x^{11} +260x^{9} -1248x^{7} +2913x^{5} +x^{4} -2922x^{3}
-7x^{2} +852x +9 $ & $[ 3 ]$ & $3|q+1$ & $2.809$ & $2$ \\ \hline
    \end{tabular}
    \caption{List of polynomials $h$ (second part).}
    \label{tab:poly_h_2}
}
\end{table}

\subsubsection*{Dimensions bigger or equal than 14}
In this subsection, we prove the following result.
\begin{proposition}\label{prop:g_ge_14}
There exist a non-$(3)$-cyclic isogeny class of $g$-dimensional abelian varieties defined over $\mathbb{F}_q$ for any prime power $q$ and any $g\geq 14$.
\end{proposition}

Here, we are going to proceed in a similar way as in \cite{HOWE2002139}, by using techniques of \cite{HSU199685}. First, we consider a polynomial of the form
\begin{align} \label{eq:h_from_T}
T_g + a_1 T_{g-1} + a_2 T_{g-2} +\dots + a_{g-1} T_1 + a_g T_0, \quad a_i\in \mathbb{Z}
\end{align}
where $T_0=1$ and for every positive integer $i$, we have $T_i(x):=2^{1+i/2} C_i(x /2^{3/2})$ is a variation of the Chebyshev polynomial $C_i \in \mathbb{Z}[x]$ defined by $C_i(\cos \theta) = \cos i\theta$. 
More precisely, we want our $h$ to be of the form
\begin{align} \label{eq:h_from_T_wanted}
T_g + a_s T_{g-s}  +\dots + a_{g-1} T_1 + a_g T_0, 
\end{align}
with $a_i \in \mathbb{Z}$ to be chosen later as well as the value of $s$. This is in order to be able to apply both results \cite[Corollary 3.2]{HSU199685} and \cite[Lemma 11]{HOWE2002139} at the same time.
As we have $T_i(x)\equiv x^{i}\pmod{2}$, from \cite[Corollary 3.2]{HSU199685}\footnote{with the notations as in the reference, we take $f(t)=t^s+a_s, r=g, l=s$ and $q=2$.} there exist an irreducible polynomial $h_2\in \mathbb{F}_2[x]$ such that it is the reduction of a polynomial $\tilde{h}\in\mathbb{Z}[x]$ over the integers of the form (\ref{eq:h_from_T_wanted}). The polynomial $\tilde{h}$ can be taken with the $a_i$ in $\{0,1\}$.
Clearly, the values $\tilde{h}'(n)$ and $\tilde{h}(n)$ ($n\in\{0,1,2\}$)\footnote{the case $q=3^r$ will require more work, which we will complete after the general case.} depends on the $a_i$. It is not hard to see that they can be modified by changing $a_{g-1}$ and $a_{g}$, respectively, since
\begin{align}\label{eq:small_Chebyshev}
T_0(x)=1,\; T_1(x)=x \quad\text{and}\quad T_2(x)=x^2-4.
\end{align}

We take as $h$ the $\tilde{h}$ but with the modifications on the two last $a_i$, as follows.
According on $\tilde{h}'(n)\equiv 0,1,2 \pmod{3}$, we add to $a_{g-1}$ one of the numbers in $\{0,2,-2\}$ in order to have $3\vert \tilde{h}'(n)$. 
After that change has been made, according on the \emph{new} $\tilde{h}(n)\equiv 0,1,\dots 8 \pmod{9}$, we add to $a_g$ (in fact, we know that $a_g=1$) the respective number in $\{0,8,-2,6,-4,4,-6,2,-8\}$. At this point, we have $3\vert \tilde{h}'(n)$ and $9\vert \tilde{h}(n)$. As we have $h_2$ irreducible (so its constant term is $1$), then $2$ is coprime with $\tilde{h}(0)$, but $\tilde{h}(0)$ may have a common divisor with $q$. If this is the case, we add to $a_g$ the number $18$ or $-18$ in order to have $\mid a_g \mid\leq 17$.
By doing as previously explained,
\begin{itemize}
\item the reductions modulo $2$ of $h$ remains $h_2$ (so $h$ is irreducible), 
\item and provided that $3\nmid q$, $h(0)$ will be coprime with $q$. 
\end{itemize}
In order to apply \cite[Lemma 11]{HOWE2002139}, we verify

\begin{align}
\left\vert \frac{a_s}{2^{s/2}} \right\vert  + \dots + \left\vert \frac{a_{g-1}}{2^{(g-1)/2}} \right\vert + \frac{1}{2} \left\vert \frac{a_g}{2^{g/2}} \right\vert <& \left\vert \frac{1}{2^{s/2}} \right\vert  + \dots + \left\vert \frac{3}{2^{(g-1)/2}} \right\vert + \frac{1}{2} \left\vert \frac{17}{2^{g/2}} \right\vert
\end{align}
which is $<1$ for $s=4$ and $g\geq 14$. Finally, both results \cite[Corollary 3.2]{HSU199685} and \cite[Lemma 11]{HOWE2002139} can be applied.

If $3\mid q$, the constant term $h(0)$ can be a multiple of $3$ and after summing $18$ or $-18$ it will still be a multiple of $3$. Remember that in this case we want $3\vert \tilde{h}'(1)$ and $9\vert \tilde{h}(1)$. As before, $\tilde{h}'(1)$ can be changed without any problem by changing $a_{g-1}$. According to the situation on $h(1)$ modulo $9$ and $h(0)$ modulo $3$, Table \ref{tab:case_q_3r} shows how to modify $a_g$ and $a_{g-2}$, respectively. This is done considering that $h(0)$ changes by the change of $a_{g} - 4a_{g-2}$ and $h(1)$ changes by the change of $a_{g} - 3a_{g-2}$, which follows from expressions in (\ref{eq:small_Chebyshev}).
At this point, we have $3\vert \tilde{h}'(1)$ and $9\vert \tilde{h}(1)$, and $3\nmid \tilde{h}(0)$. As before $2$ is also coprime with $\tilde{h}(0)$, but  $\tilde{h}(0)$ may have a common divisor with $q$. If this is the case, we add $18$ or $-18$ in order to have $\mid a_g \mid\leq 17$. By doing as previously explained,
\begin{itemize}
\item the reductions modulo $2$ of $h$ remains $h_2$ (so $h$ is irreducible), 
\item and $h(0)$ will be coprime with $q$. 
\end{itemize}

\begin{table}[ht]
    \centering
    \begin{tabular}{c|c|c|c|c|c|c|c|c|c}
         & 0 & 1 & 2 & 3 & 4 & 5 & 6 & 7 & 8  \\ \hline \hline
$0$ & $[ 6, 2 ]$ & $[ 8, 0 ]$ & $[ -2, 0 ]$ & $[ -6, 2 ]$ & $[ -4, 0 ]$ & $[ 4, 0 ]$ & $[ 0, 2 ]$ & $[ 2, 0 ]$ & $[ -8, 0 ]$ \\ \hline
$1$ & $[ 0, 0 ]$ & $[ -4, 2 ]$ & $[ -2, 0 ]$ & $[ 6, 0 ]$ & $[ 2, 2 ]$ & $[ 4, 0 ]$ & $[ -6, 0 ]$ & $[ 8, 2 ]$ & $[ -8, 0 ]$ \\ \hline
$2$ & $[ 0, 0 ]$ & $[ 8, 0 ]$ & $[ 4, 2 ]$ & $[ 6, 0 ]$ & $[ -4, 0 ]$ & $[ -8, 2 ]$ & $[ -6, 0 ]$ & $[ 2, 0 ]$ & $[ -2, 2 ]$ \\
    \end{tabular}
    \caption{How to modify the pair $[a_g, a_{g-2}]$ in the case $q=3^r$.}
    \label{tab:case_q_3r}
\end{table}

In order to apply \cite[Lemma 11]{HOWE2002139}, we verify
{\footnotesize
\begin{align}
\left\vert \frac{a_s}{2^{s/2}} \right\vert  + \dots + \left\vert \frac{a_{g-2}}{2^{(g-2)/2}} \right\vert + \left\vert \frac{a_{g-1}}{2^{(g-1)/2}} \right\vert + \frac{1}{2} \left\vert \frac{a_g}{2^{g/2}} \right\vert <& \left\vert \frac{1}{2^{s/2}} \right\vert  + \dots + \left\vert \frac{3}{2^{(g-2)/2}} \right\vert + \left\vert \frac{3}{2^{(g-1)/2}} \right\vert + \frac{1}{2} \left\vert \frac{17}{2^{g/2}} \right\vert
\end{align}
}
which is $<1$ for $s=4$ and $g\geq 14$. Finally, both results \cite[Corollary 3.2]{HSU199685} and \cite[Lemma 11]{HOWE2002139} can be applied, as in the previous case.
This completes the proof of Proposition \ref{prop:g_ge_14}.

\bibliography{ajgiangreco_EN-CAVFF}
\bibliographystyle{siam}

\end{document}